\documentclass[10pt]{article}
\usepackage{amsmath, amsthm, amssymb, amsfonts, mathrsfs, mathtools, upgreek, cite}
\usepackage{graphicx}
\usepackage{makeidx}
\usepackage[left=2cm,right=2cm,top=2cm,bottom=2cm]{geometry}
\usepackage{caption}
\usepackage{subcaption}
\usepackage[section]{placeins}
\usepackage{enumitem}
\usepackage{tikz}
\usepackage{stmaryrd}
\usepackage{hyperref}
\usepackage{upquote}

\usepackage{hyperref}
\hypersetup{colorlinks=true, citecolor={blue}, linkcolor=blue, filecolor=magenta, urlcolor=cyan}

\newcommand{\compconj}[1]{\overline{#1}}
\newtheorem*{thm}{Theorem}
\newtheorem{theorem}{Theorem}[section]

\newtheorem{lema}[theorem]{Lemma}
\newtheorem{corollary}{Corollary}[theorem]
\theoremstyle{definition}
\newtheorem{defi}{Definition}

\numberwithin{equation}{section}
\newcommand{\ip}[2]{\langle #1, #2 \rangle}
\newcommand{\divides}{\mid}

\newcommand{\cay}{\operatorname{Cay}}
\newcommand{\real}{\operatorname{Re}}

\newcommand{\iu}{\mathbf{i}}
\newcommand{\spec}{\operatorname{Spec}}
\newcommand{\lcm}{\operatorname{lcm}}

\def \Zl {{\mathbb Z}}
\def \Nl {{\mathbb N}}
\def \Rl {{\mathbb R}}
\def \Zl {{\mathbb Z}}
\def \Ql {{\mathbb Q}}
\def \Cl {{\mathbb C}}

\def \ld {{\lambda}}

\def \eu {{\textbf{e}_u}}
\def \ev {{\textbf{e}_v}}
\def \Gn {\mathcal{G}_{\Zl_n}}
\def \Ga {\mathcal{G}_{\Zl_{2^t}}}
\def \Gb {\mathcal{G}_{\Zl_{3^t}}}
\def \Gc {\mathcal{G}_{\Zl_{5^t}}}
\def \Gp {\mathcal{G}_{\Zl_{p^t}}}
\title{Periodicity and perfect state transfer of Grover walks on quadratic unitary Cayley graphs}

\author{ Koushik Bhakta and Bikash Bhattacharjya\\
	Department of Mathematics\\
	Indian Institute of Technology Guwahati, India\\
	b.koushik@iitg.ac.in, b.bikash@iitg.ac.in }

\date{}
\begin{document}
	\maketitle
	
	\vspace{-0.3in}
	
	\begin{center}{\textbf{Abstract}}\end{center}
	\noindent The quadratic unitary Cayley graph $\mathcal{G}_{\mathbb{Z}_n}$ has vertex set $\mathbb{Z}_n: =\{0,1, \hdots ,n-1\}$, where two vertices  $u$ and $v$ are adjacent if and only if $u - v$ or $v-u$ is a square of some units in $\mathbb{Z}_n$. This paper explores the periodicity and perfect state transfer of Grover walks on quadratic unitary Cayley graphs. We determine all periodic quadratic unitary Cayley graphs. From our results, it follows that there are infinitely many integral as well as nonintegral graphs that are periodic. Additionally, we determine the values of $n$ for which the quadratic unitary Cayley graph $\Gn$ exhibits perfect state transfer.

	\vspace*{0.3cm}
	
	\section{Introduction}
	Quantum walks \cite{quantum} on graphs have attracted much attention in research over the past few decades due to their potential applications in quantum simulations, quantum algorithms and quantum cryptography. There are two types of quantum walks: the continuous-time quantum walk and the discrete-time quantum walk. A discrete-time quantum walk \cite{random, godsildct} is the quantum-mechanical analogue of the classical random walk, in which a particle moves in space in discrete steps according to specific rules. In a discrete-time quantum walk, the evolution of the quantum state of a particle is described by a unitary operator, which controls how the particle moves between different positions or vertices of a graph. This paper focuses on two specific aspects of the transport properties of discrete-time quantum walks: periodicity \cite{barr} and perfect state transfer \cite{zhan}. The phenomenon of quantum states returning to their initial states at a specific time is called periodicity. This can happen because the time evolution matrix of a quantum walk is unitary. Perfect state transfer refers to the phenomenon where quantum states can be transferred between specific quantum states in a quantum system with a probability of $1$. We focus on Grover walks \cite{spec}, a particular discrete-time quantum walk. Using algebraic graph-theoretic techniques, we study the periodicity and perfect state transfer of Grover walks on quadratic unitary Cayley graphs.

	Periodicity is one of the interesting properties of quantum walks, distinguishing it from classical random walks. The periodicity property of quantum walk on a graph holds significant practical implications in physics and information science.  Periodicity could help stabilize the network and protect it from minor external disruptions. Sarkar et al.~\cite{sarkar1} studied the periodicity of quantum walks on cycles using generalized Grover coins. The periodicity of discrete-time quantum walks  on Cayley graphs was explored over the dihedral group in \cite{sarkar2} and over the symmetric group in \cite{banerjee}. Higuchi et al. \cite{higuchi1} investigated the periodicity of Grover walks on complete graphs, complete bipartite graphs and strongly regular graphs. Kubota et al. \cite{bethetrees} characterized the periodicity of Grover walks on generalized Bethe trees. Additional studies on the periodicity of Grover walks can be found in \cite{ito, kubota, bipartite, yoshie1, yoshie2, yoshie3}. Periodicity is a rare property in quantum walks, so exploring new graphs for periodicity remains an important problem. 
	
	Quadratic unitary Cayley graphs are generalizations of the well-known Paley graphs, which are extensively studied in combinatorics and number theory (see~\cite{paley_number,quc}  and the references therein). These graphs are circulant, vertex-transitive, and possess rich algebraic structures.  Circulant graphs are useful in telecommunication networks, and distributed computation (see \cite{bermond}). Compared to Cayley graphs over non-abelian groups, the spectrum of quadratic unitary Cayley graphs has a simpler description. In this paper, we focus on using spectral properties  of the graph to explore periodicity and related properties of Grover walks on the graph. Our study on these graphs gives explicit analytic results and serves as a foundation for future studies on more complex structures. The following theorem represents our first main result. For terminology and proofs, refer to later sections.

	\begin{thm}[\em{Theorem} \ref{main1}]
		The quadratic unitary Cayley graph $\Gn$ is periodic if and only if  $n$ is either $2^\alpha3^\beta$ or  $2^\delta5^\gamma$ for some nonnegative integers $\alpha$, $\beta$, $\gamma$ and $\delta$ such that $\alpha+\beta\neq 0$,  $\gamma\geq1$ and $0\leq\delta\leq2$. 
	\end{thm}
	
	Perfect state transfer in continuous-time quantum walks has been studied for the past two decades.  See \cite{periodic, polytime, pal, soffia} and the references therein. In \cite{singh}, the authors studied perfect state transfer on hypercubes and its implementation. There are recent studies of perfect state transfer on discrete-time quantum walks (see \cite{pst_symm}). \v{S}tefa\v{n}ák and Skoupý~\cite{dqw3} studied perfect state transfer in discrete-time quantum walks on complete bipartite graphs. Zhan~\cite{zhan} provided an infinite family of $4$-regular circulant graphs that exhibit perfect state transfer. Kubota and Segawa~\cite{vertextype} studied perfect state transfer of Grover walks between vertex-type states and provided a complete characterization of this quantum phenomenon on complete multipartite graphs. In \cite{bhakta1}, we studied perfect state transfer of Grover walks on unitary Cayley graphs. The following is our second main theorem, which determines all quadratic unitary Cayley graphs exhibiting perfect state transfer.
	\begin{thm}[\em{Theorem} \ref{main2}]
		The quadratic unitary Cayley graph $\Gn$ exhibits perfect state transfer if and only if $n\in\{ 2^\alpha3^\beta,~10,~20\},$ where $\alpha$ and $\beta$ are nonnegative integers such that $1\leq \alpha\leq3$ and $0\leq\beta\leq1$.
	\end{thm}

	Our work introduces new families of graphs for which periodicity and perfect state transfer can be
	fully characterized, addressing a gap in the literature, particularly in the context of the discrete-
	time quantum walks. We provided infinitely many examples of graphs that exhibit these desirable
	quantum properties.

	The paper is organized as follows. Section 2 describes several matrices related to Grover walks and provides definitions of the periodicity and perfect state transfer of Grover walks. Section 3 introduces the definition of the quadratic unitary Cayley graph and describes the eigenvalues of the quadratic unitary Cayley graph. In Section 4, we first determine when the roots of a monic rational polynomial are real parts of roots of unity. Using this, we characterize regular periodic graphs in terms of their spectrum. We then determine all periodic quadratic unitary Cayley graphs. In Section 5, we determine the values of $n$ for which $\Gn$ exhibits perfect state transfer.

	\section{Grover walk}
	Let $G$ be a graph with vertex set $V(G)$ and edge set $E(G)$. This paper considers only finite, simple, and connected graphs. We denote the elements of $E(G)$ as $uv$, where $u,v\in V$, $u\neq v$, and $uv$ and $vu$ represent the same edge. If $uv$ is an edge of a graph $G$, then the ordered pairs $(u,v)$ and $(v,u)$ are referred to as the \emph{arcs} of $G$ associated with the edge $uv$. We define $\mathcal{A}(G):=\{(u, v), (v, u):uv\in E(G) \}$, the set of all symmetric arcs of $G$. The vertices $u$ and $v$ are called the \emph{origin} and \emph{terminus}, respectively, of the arc $(u,v)$. Let $a$ be an arc of $G$. We use $o(a)$ and $t(a)$ to denote the origin and terminus of $a$, respectively. Thus for $a=(u,v)$, we have $o(a)=u$ and $t(a)=v$.  The inverse arc of $a$, denoted $a^{-1}$, is defined as the arc  $(v,u)$.

	Now, we define some matrices needed for the definition of Grover walks. The \emph{boundary matrix} $N:=N(G)\in \mathbb{C}^{V(G)\times \mathcal{A}(G)}$ of $G$ is defined by 
	$$N_{xa}=\frac{1}{\sqrt{\deg x}}\delta _{x, t(a)},$$
	where $\delta_{a,b}$ is the Kronecker delta function and $\deg x$ is the degree of the vertex $x$. The \emph{coin matrix} $C:=C(G)\in \mathbb{C}^{\mathcal{A}(G) \times \mathcal{A}(G)}$ of $G$ is defined by $$C=2N^*N-I,$$ where $I$ is the identity matrix of appropriate size. The coin matrix is also known as the Grover coin.
	The \emph{shift matrix} $R:=R(G)\in \mathbb{C}^{\mathcal{A}(G) \times \mathcal{A}(G)}$ of $G$ is defined by $$R_{ab}=\delta_{a,b^{-1}}.$$ Define the \emph{time evolution matrix} $U:=U(G)\in \mathbb{C}^{\mathcal{A}(G)\times \mathcal{A}(G)}$ of $G$ by $$U=RC.$$  A discrete-time quantum walk on a graph $G$ is characterized by a unitary matrix that operates on the complex-valued functions defined on the arcs of $G$. The discrete-time quantum walks corresponding to the matrix $U$ are commonly referred to as \emph{Grover walks}. The entries of the time evolution matrix of Grover walks can be calculated as 
	$$U_{ab} = \left\{ \begin{array}{ll}
		\frac{2}{\deg t(b)}-1 &\mbox{ if }
		a=b^{-1} \\ 
		\frac{2}{\deg t(b)} &\mbox{ if }
		t(b)=o(a)~\text{and}~a\neq b^{-1} \\
		0 &\textnormal{ otherwise.}
	\end{array}\right.$$ 
	See Figure \ref{action_u} for a visual action of the time evolution matrix of the Grover walks. Consider $\Psi\in \Cl^{\mathcal{A}(G)}$ such that the $(u,v)$-th entry of $\Psi$ is $1$, and $0$ elsewhere. The vector $\Psi$ is shown in the left graph of Figure \ref{action_u}. The vector $U\Psi$ is shown in the right graph of Figure \ref{action_u}.
	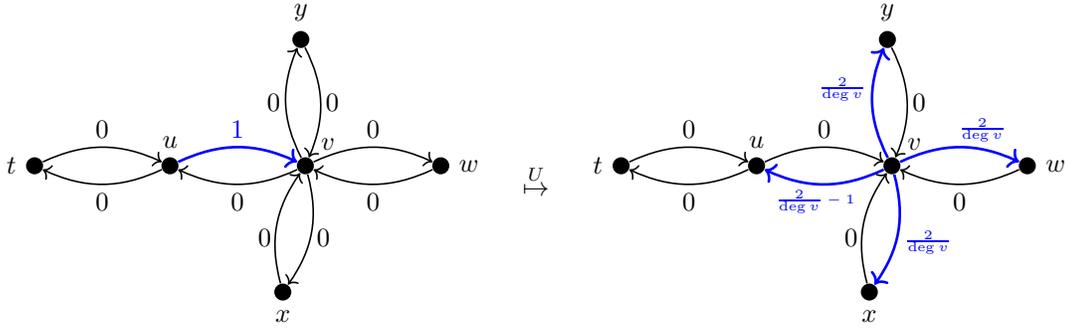
\begin{figure}[h!]
		
		\begin{center}
			\begin{tikzpicture}
				[scale = 0.6,
				slim/.style={circle,fill=black, inner sep = 0.8mm},
				]

				\node[slim] (t) at (0,0) [label = left:$t$] {};
				\node[slim] (u) at (3,0) [label = above:$u$] {};
				\node[slim] (v) at (6,0) [label = above right:$v$] {};
				\node[slim] (w) at (9,0) [label = right:$w$] {};
				\node[slim] (x) at (5.5,-2.8) [label = below:$x$] {};
				\node[slim] (y) at (5.9,2.8) [label = above:$y$] {};

				\draw[ bend left = 25, line width = .6pt][->] (u) to  (t);
				\draw[ bend left = 25, line width = .6pt][->] (t) to  (u);
				
				\draw[ bend left = 25, line width = .6pt][->] (v) to  (u);
				\draw[blue, bend left = 25, line width = 1pt][->] (u) to  (v);
				
				\draw[ bend left = 25, line width = .6pt][->] (v) to  (y);
				\draw[ bend left = 25, line width = .6pt][->] (y) to  (v);
				
				\draw[ bend left = 25, line width = .6pt][->] (x) to  (v);
				\draw[ bend left = 25, line width = .6pt][->] (v) to  (x);
				
				\draw[ bend left = 25, line width = .6pt][->] (v) to  (w);
				\draw[ bend left = 25, line width = .6pt][->] (w) to  (v);
				
				\draw (4.5,.8) node[blue] {$1$};
				\draw (4.5,-.8) node {$0$};
				\draw (1.5,.8) node {$0$};
				\draw (1.5,-.8) node {$0$};
				\draw (7.5,.8) node {$0$};
				\draw (5.1,-1.6) node {$0$};
				\draw (6.4,-1.6) node {$0$};
				\draw (7.5,-.8) node {$0$};
				\draw (5.3,1.4) node {$0$};
				\draw (6.6,1.4) node {$0$};

			\end{tikzpicture}
			\raisebox{18mm}{$\quad \overset{U}{\mapsto} \quad$}
			\begin{tikzpicture}
				[scale = 0.6,
				slim/.style={circle,fill=black, inner sep = 0.8mm},
				]

				\node[slim] (t) at (0,0) [label = left:$t$] {};
				\node[slim] (u) at (3,0) [label = above:$u$] {};
				\node[slim] (v) at (6,0) [label = above right:$v$] {};
				\node[slim] (w) at (9,0) [label = right:$w$] {};
				\node[slim] (x) at (5.5,-2.8) [label = below:$x$] {};
				\node[slim] (y) at (5.9,2.8) [label = above:$y$] {};

				\draw[ bend left = 25, line width = .6pt][->] (u) to  (t);
				\draw[ bend left = 25, line width = .6pt][->] (t) to  (u);
				
				\draw[blue, bend left = 25, line width = 1pt][->] (v) to  (u);
				\draw[ bend left = 25, line width = .6pt][->] (u) to  (v);
				
				\draw[blue, bend left = 25, line width = 1pt][->] (v) to  (y);
				\draw[ bend left = 25, line width = .6pt][->] (y) to  (v);
				
				\draw[ bend left = 25, line width = .6pt][->] (x) to  (v);
				\draw[blue, bend left = 25, line width = 1pt][->] (v) to  (x);
				
				\draw[blue, bend left = 25, line width = 1pt][->] (v) to  (w);
				\draw[ bend left = 25, line width = .6pt][->] (w) to  (v);
				
				\draw (8,.8) node[blue] {\tiny$\tfrac{2}{\deg v}$};
				\draw (4.9,1.7) node[blue] {\tiny$\tfrac{2}{\deg v}$};
				\draw (6.8,-1.7) node[blue] {\tiny$\tfrac{2}{\deg v}$};
				\draw (4.3,-.8) node[blue] {\tiny$\tfrac{2}{\deg v}-1$};
				\draw (5.1,-1.6) node {$0$};
				\draw (1.5,.8) node {$0$};
				\draw (1.5,-.8) node {$0$};
				\draw (4.5,.8) node {$0$};
				\draw (7.5,-.8) node {$0$};
				\draw (6.6,1.4) node {$0$};
			\end{tikzpicture}
			\caption{Action of $U$} \label{action_u}
		\end{center} 	
	\end{figure}
	
	The \emph{discriminant} $P:=P(G)\in \mathbb{C}^{V(G)\times V(G)}$ of $G$ is defined by $$P=NRN^*.$$
	We refer the reader to \cite{qq} for more details about the matrices $N,~R,~U$ and $P$. The adjacency matrix $A:=A(G)\in \Cl^{V(G)\times V(G)}$ of $G$ is defined by
	$$A_{uv} = \left\{ \begin{array}{rl}
		1 &\mbox{ if }
		uv\in E(G) \\ 
		0 &\textnormal{ otherwise.}
	\end{array}\right.$$   If  $G$ is a regular graph, then the matrices $P$ and $A$ are closely related.
	\begin{lema}\emph{\cite{qq}}\label{reg}
		Let $A$ and $P$ be the adjacency and discriminant matrices of a $k$-regular graph, respectively. Then $P=\frac{1}{k}A$. 
	\end{lema}
	For a matrix $M$ associated to a graph $G$, we define $\spec_M(G)$ as the set of all distinct eigenvalues of $M$. A graph $G$ is called \emph{integral} if all the eigenvalues of its adjacency matrix are integers; otherwise, it is called \emph{nonintegral}.
	\begin{defi}  
		The Grover walk on a graph $G$ is \emph{periodic} if there exists a positive integer \( \tau \) such that the time evolution operator $U$ of $G$ satisfies $U^\tau = I$. 
	\end{defi}  
	For brevity, we say $G$ is periodic to mean that the Grover walk on $G$ is periodic.   The smallest positive integer $\tau$ such that $U^\tau =I$ is termed the \emph{period} of $G$, and $G$ is called \emph{$\tau$-periodic}.

	Note that $U$ is a diagonalizable matrix. Hence, the periodicity of a graph can be easily found by the eigenvalues of its time evolution matrix.
	\begin{lema}\emph{\cite{mixedpaths}} \label{period}
		A graph $G$ is a $\tau$-periodic graph if and only if $ \eta^\tau  =1$ for every $\eta \in \spec_U(G)$, and  there is some $\eta\in\spec_U(G)$ such that $\eta^j\neq 1$ for each  $j\in\{1,\hdots,\tau-1\}$.
	\end{lema}
	From the previous lemma, we observe that the period of a periodic graph $G$ can be determined from the eigenvalues of $U$.
	\begin{corollary} \emph{\cite{bhakta1}}\label{periodic}
		Let $\eta_1, \hdots, \eta_t$ be the distinct eigenvalues of the time evolution matrix of a periodic graph $G$. Let $k_1, \hdots, k_t$ be the least positive integers such that $\eta_1 ^{k_1}=1, \hdots,\eta_t^{k_t}=1$. Then $G$ is periodic and the period of $G$ is $\lcm(k_1, \hdots, k_t)$.
	\end{corollary}
	For a nonnegative integer $k$, $\{a\}^k$ denotes the multi-set $\{a,\hdots,a\}$, where the element $a$ repeats $k$ times. Additionally, $\sqrt{-1}$ is denoted by $\iu$. The spectral mapping theorem of Grover walks states the following result.
	
	\begin{lema} \emph{\cite{higu}} \label{evu}
		Let $\mu_1, \hdots,\mu_n$ be the eigenvalues of the discriminant of a graph $G$. Then, the multi-set of eigenvalues of the time evolution matrix is $$\left\{e^{\pm \iu \arccos(\mu_j)}:j\in\{1,\hdots,n\}\right\}\cup \{1\}^{b_1} \cup \{ -1\}^{b_1-1+1_B},$$ where $b_1=|E(G)|-|V(G)|+1$, and $1_B=1$ or $0$ according as $G$ is bipartite or not.	
	\end{lema}
	

	For $\Phi,\Uppsi\in \Cl^{\mathcal{A}(G)}$, the Euclidean inner product of $\Phi$ and $\Uppsi$ is $\ip{\Phi}{\Uppsi}$. Let $G$ be a graph and $U$ be its time evolution matrix. A vector $\Phi\in \Cl^{\mathcal{A}(G)}$ is called a \emph{state} if $\ip{\Phi}{\Phi}=1$. 
	
	We say \emph{perfect state transfer} occurs in $G$ from a state $\Phi$ to another state $\Psi$ at time $\tau\in \Nl$ if there exists a unimodular complex number $\gamma$ such that $U^\tau\Phi=\gamma \Psi.$ See~\cite{zhan} for more details.

	A state $\Phi\in \Cl^{\mathcal{A}(G)}$ is called a \emph{vertex-type} of a graph $G$ if there exists a vertex $u$ such that $\Phi=N^* \eu$, where $\eu$ is the unit vector defined by $(\eu)_x=\delta_{u,x}$. We denote $\Sigma$ as the set of all vertex-type states, that is, $\Sigma =\{N^*\eu: u\in V(G)\}$.  In this paper, we consider only vertex-type states. We refer the reader to \cite{vertextype} for more details about vertex-type states.
	\begin{defi}
		A graph exhibits \emph{perfect state transfer} from vertex \( u \) to vertex \( v \) at time \( \tau \) if there exists a unimodular complex number \( \gamma \) such that the graph exhibits perfect state transfer from the state $N^*\eu$ to the state $N^*\ev$, that is, $U^\tau N^*\eu=\gamma N^*\ev$.
	\end{defi}

	The next lemma is a consequence of Lemma \ref{evu} and Corollary \ref{periodic}. This result helps in characterizing the periodicity of a graph based on its discriminant eigenvalues. For a complex number $z$, we denote $\real(z)$ as the real part of $z$, that is, $\real(z)=\frac{z+\bar{z}}{2}$. Let $\Re$ be the set of the real parts of roots of unity, that is, $\Re=\left\{\real(z): z\in \Cl~\text{and}~z^n=1 ~\text{for some positive integer $n$}\right\}$.

	\begin{lema}\label{m1}
		Let $G$ be a graph with discriminant $P$. Then, $G$ is periodic if and only if $\spec_P(G)\subset\Re$. 
	\end{lema}
	\begin{proof}
		Let $\mu\in\spec_P(G)$. Then by	Lemma 2.3 and Corollary 2.2.1, $G$ is periodic if and only if $\left(e^{\pm \iu \arccos \mu}\right)^\tau =1$ for some positive integer $\tau$. Therefore, $G$ is periodic if and only if $\mu=\cos \frac{2k\pi}{\tau}$ for some integers $k$ and $\tau(\geq 1)$. Hence, the result follows.
	\end{proof}
	
	\section{Quadratic unitary Cayley graphs}
	This section introduces the definitions of quadratic unitary Cayley graphs and outlines key preliminary concepts and results required for the subsequent sections. 	Let $(\Gamma,+)$ be a finite abelian group and let $S$ be an inverse closed subset of $\Gamma\setminus\{0\}$, where $0$ is the identity element of $\Gamma$.
	\begin{defi}
		A \emph{Cayley graph} $\cay(\Gamma, S)$ of the group $\Gamma$ with respect to $S$ is an undirected graph whose vertex set is $\Gamma$ and two vertices $u$ and $v$ are adjacent if and only if $u-v\in S$. 
	\end{defi}
	
	If $\Gamma=\Zl_{n}$, then the Cayley graph is called a \emph{circulant graph}.  Theoretically, the eigenvalues of the adjacency matrix of $\cay(\Gamma, S)$ can be determined by the characters of the group $\Gamma$, see \cite{rep} for details. A character of an abelian group $\Gamma$ is a homomorphism from $\Gamma$ to the multiplicative group of complex numbers. It is well known that for each $a\in \Zl_n$, the mapping $\chi_a:\Zl_n\rightarrow\Cl^*$ defined by  $$\chi_a(x)=e^{-\frac{2\pi ax\iu}{n}} ~~\text{for}~ x\in \Zl_n,$$ is a character of $\Zl_n$. 
	The eigenvalues $\lambda_a$ of the adjacency matrix (see \cite{rep})  of the circulant graph $\cay(\Zl_n,S)$ are given by  \begin{equation}\label{evc}
		\lambda_a=\chi_a(S)=\sum_{s\in S}\chi_a(s)~~\text{for}~~a\in\Zl_n.
	\end{equation}  
	
	Consider $\Zl_n$ as the ring of integers modulo $n$. Let $\Zl_n^*=\{a\in\Zl_n:\gcd(a,n)=1\}$, the set of units of the ring $\Zl_n$. The unitary Cayley graph \cite{ucg} is a well-known Cayley graph defined as $\cay(\Zl_n,\Zl_n^*)$. Beaudrap \cite{qucfirst} introduced another family of Cayley graphs known as quadratic unitary Cayley graphs. Let $Q_n=\{r^2:r\in \Zl_n^*\}$ and $V_n=Q_n\cup (-Q_n)$ (not necessarily disjoint). The Cayley graph $\cay(\Zl_n, V_n)$ is known as the \emph{quadratic unitary Cayley graph}. We prefer to denote the graph $\cay(\Zl_n,V_n)$ by $\Gn$. Note that if $n\equiv1 \pmod 4$ is a prime number, then $\Gn$ is the well-known Paley graph. See \cite{liu} for more details about quadratic unitary Cayley graphs. Figure~\ref{fig2} visualizes the quadratic unitary Cayley graphs defined over the groups $\Zl_{9}$ and $\Zl_{20}$.
	
\begin{figure}[h!]
	\centering
	\begin{subfigure}{.5\textwidth}
		\centering
		\begin{tikzpicture}[scale=1.7] 
			\node[circle, draw, fill=black, inner sep=2pt] (0) at ({360/9*0}:2) {};
			\node[circle, draw, fill=black, inner sep=2pt] (1) at ({360/9*1}:2) {};
			\node[circle, draw, fill=black, inner sep=2pt] (2) at ({360/9*2}:2) {};
			\node[circle, draw, fill=black, inner sep=2pt] (3) at ({360/9*3}:2) {};
			\node[circle, draw, fill=black, inner sep=2pt] (4) at ({360/9*4}:2) {};
			\node[circle, draw, fill=black, inner sep=2pt] (5) at ({360/9*5}:2) {};
			\node[circle, draw, fill=black, inner sep=2pt] (6) at ({360/9*6}:2) {};
			\node[circle, draw, fill=black, inner sep=2pt] (7) at ({360/9*7}:2) {};
			\node[circle, draw, fill=black, inner sep=2pt] (8) at ({360/9*8}:2) {};

			\node[left, xshift=-2pt, yshift=-.5pt] at (5) {$6$};
			\node[above left, xshift=0pt, yshift=-2pt] at (4) {$7$};
			\node[above left] at (3) {$8$};
			\node[above, yshift=2pt] at (2) {$0$}; 
			\node[above right, xshift=1pt, yshift=-1pt] at (1) {$1$};
			\node[right , xshift=2pt, yshift=0pt] at (0) {$2$};
			\node[below left] at (6) {$5$};
			\node[below , xshift=0pt, yshift=-2pt] at (7) {$4$};
			\node[below right, xshift=0pt] at (8) {$3$};

			\node[below, xshift=0pt, yshift=-25pt] at (7) {$V_{9}=\{1,2,4,5,7,8\}$}; 
			
			\draw (0) -- (1);
			\draw (2) -- (1);
			\draw (0) -- (8);
			\draw (2) -- (3);
			\draw (3) -- (4);
			\draw (4) -- (5);
			\draw (5) -- (6);
			\draw (6) -- (7);
			\draw (7) -- (8);
			\draw (0) -- (2);
			\draw (0) -- (4);
			\draw (0) -- (5);
			\draw (0) -- (7);
			\draw (1) -- (3);
			\draw (1) -- (5);
			\draw (1) -- (6);
			\draw (1) -- (8);
			\draw (2) -- (4);
			\draw (2) -- (6);
			\draw (2) -- (7);
			\draw (3) -- (5);
			\draw (3) -- (7);
			\draw (3) -- (8);
			\draw (4) -- (6);
			\draw (4) -- (8);
			\draw (5) -- (7);
			\draw (6) -- (8);

			
		\end{tikzpicture}
		\caption{$\mathcal{G}_{\mathbb{Z}_{9}}$}\label{fig:sub1}
	\end{subfigure}%
	\begin{subfigure}{.5\textwidth}
		\centering

		
		\begin{tikzpicture}[scale=1.7] 
			\node[circle, draw, fill=black, inner sep=2pt] (0) at ({360/20*0}:2) {};
			\node[circle, draw, fill=black, inner sep=2pt] (1) at ({360/20*1}:2) {};
			\node[circle, draw, fill=black, inner sep=2pt] (2) at ({360/20*2}:2) {};
			\node[circle, draw, fill=black, inner sep=2pt] (3) at ({360/20*3}:2) {};
			\node[circle, draw, fill=black, inner sep=2pt] (4) at ({360/20*4}:2) {};
			\node[circle, draw, fill=black, inner sep=2pt] (5) at ({360/20*5}:2) {};
			\node[circle, draw, fill=black, inner sep=2pt] (6) at ({360/20*6}:2) {};
			\node[circle, draw, fill=black, inner sep=2pt] (7) at ({360/20*7}:2) {};
			\node[circle, draw, fill=black, inner sep=2pt] (8) at ({360/20*8}:2) {};
			\node[circle, draw, fill=black, inner sep=2pt] (9) at ({360/20*9}:2) {};
			\node[circle, draw, fill=black, inner sep=2pt] (10) at ({360/20*10}:2) {};
			\node[circle, draw, fill=black, inner sep=2pt] (11) at ({360/20*11}:2) {};
			\node[circle, draw, fill=black, inner sep=2pt] (12) at ({360/20*12}:2) {};
			\node[circle, draw, fill=black, inner sep=2pt] (13) at ({360/20*13}:2) {};
			\node[circle, draw, fill=black, inner sep=2pt] (14) at ({360/20*14}:2) {};
			\node[circle, draw, fill=black, inner sep=2pt] (15) at ({360/20*15}:2) {};
			\node[circle, draw, fill=black, inner sep=2pt] (16) at ({360/20*16}:2) {};
			\node[circle, draw, fill=black, inner sep=2pt] (17) at ({360/20*17}:2) {};
			\node[circle, draw, fill=black, inner sep=2pt] (18) at ({360/20*18}:2) {};
			\node[circle, draw, fill=black, inner sep=2pt] (19) at ({360/20*19}:2) {};

			\node[above, xshift=0pt, yshift=2pt] at (5) {$0$};
			\node[above right, xshift=-2pt, yshift=1pt] at (4) {$1$};
			\node[above right] at (3) {$2$};
			\node[above right] at (2) {$3$}; 
			\node[above right, xshift=1pt, yshift=-1pt] at (1) {$4$};
			\node[right , xshift=2pt, yshift=0pt] at (0) {$5$};
			\node[above left] at (6) {$19$};
			\node[above left] at (7) {$18$};
			\node[above left] at (8) {$17$};
			\node[left, xshift=-1pt, yshift=0pt] at (9) {$16$}; 
			\node[below left, xshift=-2pt, yshift=1pt] at (10) {$15$};
			\node[below left] at (11) {$14$};
			\node[below left] at (12) {$13$};
			\node[below left] at (13) {$12$};
			\node[below left, xshift=1pt, yshift=-1pt] at (14) {$11$};
			\node[below , xshift=0pt, yshift=-2pt] at (15) {$10$};
			\node[below right] at (16) {$9$};
			\node[below right] at (17) {$8$};
			\node[below right] at (18) {$7$};
			\node[below right] at (19) {$6$};

			\node[below, xshift=0pt, yshift=-25pt] at (15) {$V_{20}=\{1,9,11,19\}$}; 
			
			\draw (0) -- (1);
			\draw (2) -- (1);
			\draw (0) -- (19);
			\draw (2) -- (3);
			\draw (3) -- (4);
			\draw (4) -- (5);
			\draw (5) -- (6);
			\draw (6) -- (7);
			\draw (7) -- (8);
			\draw (8) -- (9);
			\draw (9) -- (10);
			\draw (10) -- (11);
			\draw (11) -- (12);
			\draw (12) -- (13);
			\draw (13) -- (14);
			\draw (14) -- (15);
			\draw (15) -- (16);
			\draw (16) -- (17);
			\draw (17) -- (18);
			\draw (18) -- (19);   
			\draw (5) -- (14);
			\draw (5) -- (16);
			\draw (4) -- (13);
			\draw (4) -- (15);
			\draw (2) -- (11);
			\draw (2) -- (13);
			\draw (1) -- (10);
			\draw (1) -- (12);
			\draw (19) -- (8);
			\draw (19) -- (10);
			\draw (18) -- (7);
			\draw (18) -- (9);
			\draw (6) -- (15);
			\draw (7) -- (16);
			\draw (3) -- (14);
			\draw (3) -- (12);
			\draw (0) -- (11);
			\draw (0) -- (9);
			\draw (17) -- (6);
			\draw (17) -- (8);
		\end{tikzpicture}
		\caption{$\mathcal{G}_{\mathbb{Z}_{20}}$}
		\label{fig:sub2}
	\end{subfigure}
	\caption{Examples of quadratic unitary Cayley graphs }
	\label{fig2}
\end{figure}

	From \eqref{evc}, the eigenvalues of the adjacency matrix of $\Gn$ can be expressed as  \begin{equation}\label{eev}
		\lambda_a=\sum_{r\in V_n} e^{-\frac{2\pi ar\iu}{n}}~~\text{for}~~a\in \Zl_n.
	\end{equation}
	Huang \cite{quc} gave an explicit form of the adjacency eigenvalues of the quadratic unitary Cayley graphs. The expressions of the eigenvalues of $\Gn$ given in Haung \cite{quc} are presented in Theorem \ref{quc2}.  The next two results describe some fundamental properties of $Q_n$.
	\begin{lema}\emph{\cite{zucker}}\label{lp}
		Let $t$ be a positive integer and $p$ be a prime number.
		\begin{enumerate}[label=(\roman*)]
			\item If $p=2$, then $-1\notin Q_{2^t}$ for $t\geq 2$.
			\item If $p\equiv 1\pmod 4$, then $-1\in Q_{p^t}$. 
			\item If $p\equiv 3\pmod 4$, then $-1\notin Q_{p^t}$.
		\end{enumerate}
	\end{lema}
	The following lemma helps in computing $|V_n|$, the regularity of $\Gn$.
	\begin{lema}\label{kkk}
		Let $t$ be a positive integer and $p$ be a prime number. Then,
		 \begin{enumerate}[label=(\roman*)]
			\item     $|Q_{2^t}|=\left\{ \begin{array}{cc}
				1 &\mbox{ if } t\leq 2 \\
				2^{t-3} &\mbox{ if } t\geq 3.
			\end{array} \right.$         
			\item $|Q_{p^t}|=\frac{p^{t-1}(p-1)}{2}$ for $p\geq3$. 
		\end{enumerate}
	\end{lema}
	\begin{proof}
		\begin{enumerate}[label=(\roman*)]
			\item It is clear that $|Q_2|=1=|Q_4|$. Let $t\geq 3$. Now, consider the group homomorphism $\psi:\Zl_{2^t}^*\rightarrow Q_{2^t}$ defined by $\psi(x)=x^2$ for $x\in \Zl_{2^t}^* $. It is clear that $\ker(\psi)=\left\{1, -1, 1+2^{t-1}, -1+2^{t-1}\right\}$ and $\psi$ is surjective. Therefore, by the first homomorphism theorem, the result follows.
			\item For $p\geq3$, consider the surjective group homomorphism $\psi:\Zl_{p^t}^*\rightarrow Q_{p^t}$ defined by  $\psi(x)=x^2$ for $x\in \Zl_{p^t}^* $. We find $\ker(\psi)=\{1, -1\}$. Hence, the result follows. \qedhere
		\end{enumerate}
	\end{proof}
	Let $p$ be an odd prime. An integer $a$ is a quadratic residue modulo $p$  if it is congruent to a perfect square modulo $p$; otherwise, $a$ is a quadratic non-residue modulo $p$. The \emph{Legendre symbol} is a function of $a$  and $p$, denoted $\left(a\Big/p\right)$, where
	$$\left(a\Big/p\right)=\left\{ \begin{array}{rl}
		1 &\mbox{ if } \text{$a$ is a quadratic residue modulo $p$ and $a\not\equiv 0 \pmod p$}, \\ -1 & \mbox{ if } \text{$a$ is a quadratic non-residue modulo $p$,}\\ 0 & \mbox{ if } a\equiv 0\pmod p.
	\end{array}\right.$$
	
	See \cite{zucker} for more details about the Legendre symbol. Some basic properties of the Legendre symbol are given in the following theorem.
	\begin{theorem}\emph{\cite{zucker}}
		Let $p$ be an odd prime, and let $a$ and $b$ be integers relatively prime to $p$. Then, the following properties are satisfied by the Legendre symbol.
		\begin{enumerate}[label=(\roman*)]
			\item If $a\equiv b \pmod p$, then $\left(a\Big/p\right)=\left(b\Big/p\right)$.
			\item $\left(a^2\Big/p\right)=1$.
			\item $\left(ab\Big/p\right)=\left(a\Big/p\right)\left(b\Big/p\right)$. 
			\item $\left(1\Big/p\right)=1$ and $\left(-1\Big/p\right)=(-1)^{\frac{p-1}{2}}$. 
		\end{enumerate}
	\end{theorem}
	For each positive integer $n$, the \emph{quadratic Gaussian sum} of order $n$ is the complex-valued function $G_n:\Zl \rightarrow \Cl$ defined by the equation 
	\begin{equation}\label{defgauss}
		G_n(a)=\frac{1}{\sqrt{n}}\sum_{k\in \Zl_n} e^{-\frac{2\pi k^2a\iu}{n}}.
	\end{equation}
	Quadratic Gaussian sum is introduced in \cite{gauss}. It is proved in \cite{gauss} that some quadratic Gaussian sums are the eigenvalues of the Fourier transform.
	\begin{lema}\emph{\cite{gauss}}\label{kk}
		Let $p$ be a prime number. Then, $$G_p(1)=\left\{ \begin{array}{rr}
			1 &\mbox{ if } p\equiv 1 \pmod 4 \\
			-\iu &\mbox{ if } p\equiv 3 \pmod 4.
		\end{array} \right.$$
	\end{lema}
	
	Huang \cite{quc} used some properties of quadratic Gaussian sums to determine the value of $\chi_a(Q_{p^t})$ in the following theorem.
	\begin{theorem}\emph{\cite{quc}}\label{quc1}
		Let  $p$ be a prime number. Suppose $t$ and $a$ be positive integers such that $a\in \Zl_{p^t}$.
		\begin{enumerate}[label=(\roman*)]
			\item If $p=2$, then $$\chi_a(Q_2)=\sum_{r\in Q_2} \chi_a(r)=\cos(a\pi),~~~\chi_a(Q_4)=\sum_{r\in Q_4} \chi_a(r)=e^{-\frac{a\pi \iu}{2}}$$ and if $t\geq 3$, then  $$\chi_a(Q_{2^t})=\sum_{j\in Q_{2^t}} \chi_a(j)=\left\{ \begin{array}{ll}
				2^{t-3}e^{-\frac{a\pi \iu}{2^{t-1}}} &\mbox{ if } a\in\left\{2^{t-3}, 3\cdot 2^{t-3}, 5\cdot2^{t-3}, 7\cdot 2^{t-3} \right\}, \\ 2^{t-3}(-\iu)^{\frac{a}{2^{t-2}}} & \mbox{ if } a\in \left\{0, 2^{t-2}, 2^{t-1}, 3\cdot2^{t-2} \right\}, \\
				0 & \text{ otherwise.}		
			\end{array}\right. $$
			\item If $p\geq 3$ and $t\geq1$, then $$\chi_a(Q_{p^t})=\sum_{j\in Q_{p^t}}\chi_a(j)= \left\{ \begin{array}{ll}
				\frac{p^{t-1}(p-1)}{2}&\mbox{ if } a=0, \\ \frac{1}{2}p^{t-1}\left[ \sqrt{p}\left( \frac{a}{p^{t-1}}\Big/ p\right) G_p(1)-1\right] & \mbox{ if } a\in p^{t-1} \Zl_p\setminus \{0\}, \\
				0 & \text{ otherwise.}
			\end{array}\right. $$
		\end{enumerate}
		
	\end{theorem}
	The following theorem gives the eigenvalues $\ld_a$ given in Equation \eqref{eev} of the adjacency matrix of the quadratic unitary Cayley graph.
	\begin{theorem}\emph{\cite{quc}}\label{quc2}
		Let $p_1^{t_1}p_2^{t_2}\hdots p_s^{t_s}$ be the prime factorization of a positive integer $n$ such that $-1\in Q_{p_k^{t_k}}$ for $k\in\{1,\hdots, r\}$ and $-1\notin Q_{p_k^{t_k}}$ for $k\in\{r+1, \hdots, s\},$ where $0\leq r\leq s$. Let $a\in\Zl_n$ be identified with $a=(a_1,\hdots,a_s)\in\Zl_{p_1^{t_1}}\times\cdots\times\Zl_{p_s^{t_s}}$.
		\begin{enumerate}[label=(\roman*)]
			\item If $r=s$, then the eigenvalues of the adjacency matrix of $\Gn$ are given by $$\ld_a=\chi_{a_1}\left(Q_{p_1^{t_1}}\right)\hdots\chi_{a_s}\left(Q_{p_s^{t_s}}\right)
			~~\text{for}~a\in\Zl_n.$$
			\item If $r<s$, then the eigenvalues of the adjacency matrix of $\Gn$ are given by $$\ld_a=\chi_{a_1}\left(Q_{p_1^{t_1}}\right)\hdots \chi_{a_r} \left(Q_{p_r^{t_r}}\right)  \left[ \chi_{a_{r+1}}\left(Q_{p_{r+1}^{t_{r+1}}}\right)\hdots \chi_{a_s} \left(Q_{p_s^{t_s}}\right)+ \compconj{\chi_{a_{r+1}}\left(Q_{p_{r+1}^{t_{r+1}}}\right) } \hdots \overline{\chi_{a_s} \left(Q_{p_s^{t_s}}\right)} \right] $$ $\text{for}~a\in\Zl_n$.
		\end{enumerate}
		
	\end{theorem}
	Using Theorems \ref{quc1} in	Theorem \ref{quc2}, one can explicitly calculate the eigenvalues $\ld_a$ of the adjacency matrix of $\Gn$. In the following two sections, we use the explicit expression of $\ld_a$ in the study of periodicity and perfect state transfer of $\Gn$.
	
	\section{Periodicity}
	Let $n$ be a positive integer. The $n$-th \emph{cyclotomic polynomial} $\Phi_n(x)$ is the monic polynomial defined by $$\Phi_n(x)=\prod_{\xi} (x-\xi),$$ where the product runs over all primitive complex $n$-th roots of unity. See \cite{galois} for more information about cyclotomic polynomials.
	\begin{theorem}\emph{\cite{galois}}\label{alg1}
		Let $\Phi_n(x)$ be the $n$-th cyclotomic polynomial. Then, the following are true.
		\begin{enumerate}[label=(\roman*)]		
			\item $\Phi_n(x)$ is an irreducible polynomial in $\Zl[x]$.
			\item $x^n-1=\prod\limits_{m|n}\Phi_m(x)$.
		\end{enumerate}
	\end{theorem}
	The following lemma is a direct consequence of Theorem \ref{alg1}.
	\begin{lema}\label{m2}
		Let $f(x)$ be a monic polynomial of degree $n$ with coefficients in $\Ql$. Then, the solutions of $f(x)=0$ are the real parts of some roots of unity if and only if the polynomial $(2x)^nf\left( (x+x^{-1})/2\right)$ is a product of some cyclotomic polynomials.
	\end{lema}
	\begin{proof}
		Let $g(x)=(2x)^nf\left( (x+x^{-1})/2\right)$.  Write $f(x)=(x-\alpha_1)\cdots(x-\alpha_n)$, where $\alpha_1,\hdots, \alpha_n\in \Cl$. Then, we have $g(x)=(x^2-2\alpha_1x+1)\cdots(x^2-2\alpha_nx+1)$. Note that $g(x)$ is a monic polynomial in $\Ql[x]$. It is easy to observe that the solutions of $f(x)=0$ are the real parts of some roots of unity if and only if the solutions of $g(x)=0$ are roots of unity. Since $\Ql[x]$ is a unique factorization domain, $g(x)$ can be factorized into irreducible polynomials over $\Ql$. Therefore by Theorem \ref{alg1},  the result follows.
	\end{proof}
	
	\begin{lema}\emph{\cite{complex}}\label{root}
		Let $f(x)$ be a polynomial with coefficients in $\Ql$. If $a+\sqrt{b}$ is a root of $f(x)$, where $a,b\in \Ql$ and $b$ is not a square, then $a-\sqrt{b}$ is also a root of $f(x)$, with the same multiplicity.
	\end{lema}

	Let $\Delta=\{a\pm\sqrt{b}:a,b\in \Ql~\text{and}~b ~\text{is not a square}\}$, and $\overline{\Delta}=\Rl\setminus(\Ql\cup\Delta)$. For a subset $F$ of real numbers, let $\spec_P^F(G)=F\cap\spec_P(G)$.
	
	\begin{theorem}\label{ls}
		Let $G$ be a regular graph with discriminant matrix $P$. Then, $G$ is periodic if and only if $$\spec_P^\Ql(G)\subseteq\left\{\pm1,\pm\frac{1}{2},0\right\},~\spec_P^\Delta(G)\subseteq\left\{\pm \frac{\sqrt{3}}{2}, \pm\frac{1}{4}\pm\frac{\sqrt{5}}{4},\pm\frac{1}{\sqrt{2}}\right\}~\text{and}~\spec_P^{\overline{\Delta}}(G)\subset \Re.$$
	\end{theorem}
	\begin{proof}
		Let $G$ be a $k$-regular periodic graph. From Lemma \ref{m1}, we know that $\spec_P(G)\subset\Re$, and hence $\spec_P^{\overline{\Delta}}\subset \Re$. Now let $\mu\in\spec_P^\Ql(G)$ and consider $f(x)=x-\mu$. Since $G$ is periodic and $f(x)\in \Ql[x]$, by Lemma \ref{m2}, $(2x)f\left(\frac{x+x^{-1}}{2}\right)$ is a product of cyclotomic polynomials. Note that 
		$$(2x)f\left(\frac{x+x^{-1}}{2}\right)=x^2-2\mu x+1.$$ 
		Further, the cyclotomic polynomials of degree at most $2$ are $\Phi_1(x)$, $\Phi_2(x)$, $\Phi_3(x)$, $\Phi_4(x)$ and $\Phi_6(x)$. Since $x^2-2\mu x +1$ is a product of cyclotomic polynomials, we find that $\mu\in\left\{\pm1,\pm\frac{1}{2},0\right\}.$

		Now, let $a+\sqrt{b}\in \spec_P^\Delta(G)$. Since $G$ is a $k$-regular graph,  $P=\frac{1}{k}A$, where $A$ is the adjacency matrix of $G$. So, the characteristic polynomial of $P$ has rational coefficients. Therefore by Lemma  \ref{root}, we have $a-\sqrt{b}\in\spec_P^\Delta(G)$. Now, let $f(x)=(x-(a+\sqrt{b}))(x-(a-\sqrt{b}))$. Then, $f(x)$ is a monic polynomial with rational coefficients and   
		$$(2x)^2f\left(\frac{x+x^{-1}}{2}\right)=x^4-4ax^3+(4(a^2-b)+2)x^2-4ax+1.$$ 
		Since $G$ is periodic, by Lemma \ref{m2},   $x^4-4ax^3+(4(a^2-b)+2)x^2-4ax+1$ is a product of cyclotomic polynomials. One can observe that this polynomial cannot be expressed as a product of polynomials from $\left\{\Phi_1(x), \Phi_2(x), \Phi_3(x),\Phi_4(x), \Phi_6(x)\right\}$. Therefore, this polynomial must be a cyclotomic polynomial of degree $4$. Thus,
		$$(2x)^2f\left(\frac{x+x^{-1}}{2}\right)\in\left\{\Phi_5(x),\Phi_8(x),\Phi_{10}(x),\Phi_{12}(x)\right\}.$$
		Comparing $x^4-4ax^3+(4(a^2-b)+2)x^2-4ax+1$ with $\Phi_5(x),~ \Phi_8(x),~ \Phi_{10}(x)~\text{and}~\Phi_{12}(x)$, we find that $$a\pm \sqrt{b}\in\left\{\pm\frac{\sqrt{3}}{2}, \pm\frac{1}{4}\pm\frac{\sqrt{5}}{4},\pm\frac{1}{\sqrt{2}}\right\}.$$ 
		
		Conversely, assume that $$\spec_P^\Ql(G)\subseteq\left\{\pm1,\pm\frac{1}{2},0\right\},~\spec_P^\Delta(G)\subseteq\left\{\pm \frac{\sqrt{3}}{2}, \pm\frac{1}{4}\pm\frac{\sqrt{5}}{4},\pm\frac{1}{\sqrt{2}}\right\}~\text{and}~\spec_P^{\overline{\Delta}}(G)\subset \Re.$$
		Then, $\spec_P(G)\subset \Re$. Therefore by Lemma \ref{m1}, $G$ is periodic.
	\end{proof}

	\begin{lema}\label{main}
		Let $p$ be a prime factor of $n$. 	If $\Gn$ is periodic, then $p\in\{2,3,5\}$.
	\end{lema}
	\begin{proof}
		Let $p_1^{t_1}p_2^{t_2}\hdots p_s^{t_s}$ be the prime factorization of $n$ such that $-1\in Q_{p_k^{t_k}}$ for $k\in\{1,\hdots, r\}$ and $-1\notin Q_{p_k^{t_k}}$ for $k\in\{r+1, \hdots, s\},$ where $0\leq r\leq s$. We identify $\Zl_n$ with 
		\begin{equation*}
			\Zl_n=\Zl_{p_1^{t_1}}\times \Zl_{p_2^{t_2}}\times \cdots \times \Zl_{p_s^{t_s}}.\end{equation*} 
		With this identification, we have
		$$	Q_n=\left( Q_{p_1^{t_1}},\hdots,Q_{p_r^{t_r}},Q_{p_{r+1}^{t_{r+1}}}, \hdots, Q_{p_s^{t_s}}\right)~\text{and}~ -Q_n=\left( Q_{p_1^{t_1}},\hdots,Q_{p_r^{t_r}},-Q_{p_{r+1}^{t_{r+1}}}, \hdots, -Q_{p_s^{t_s}}\right) .$$
		If $r=s$ then $Q_n=-Q_n$, and if $r<s$ then $Q_n\cap (-Q_n)=\emptyset$.	 First, assume that $r<s$. Since  $Q_n\cap (-Q_n)=\emptyset$, we find that the regularity of the graph $\Gn$ is $2|Q_n|$, which is equal to  $2 |Q_{p_1^{t_1}}|\hdots |Q_{p_s^{t_s}}|$.
		
		Let $a=(a_1,a_2,\hdots, a_s)\in \Zl_n$ for some $a_i\in \Zl_{p_i^{t_i}}$, where $1\leq i\leq s$. By Theorem \ref{quc2}, the eigenvalues of the adjacency matrix of $\mathcal{G}_{\Zl_n}$ are given by \begin{equation*}
			\ld_a=\chi_{a_1}\left(Q_{p_1^{t_1}}\right)\hdots \chi_{a_r} \left(Q_{p_r^{t_r}}\right)  \left[ \chi_{a_{r+1}}\left(Q_{p_{r+1}^{t_{r+1}}}\right)\hdots \chi_{a_s} \left(Q_{p_s^{t_s}}\right)+ \compconj{\chi_{a_{r+1}}\left(Q_{p_{r+1}^{t_{r+1}}}\right) } \hdots \compconj{\chi_{a_s} \left(Q_{p_s^{t_s}}\right)} \right]	
		\end{equation*}
		for each $a\in \Zl_n$. Let $p$ be a prime such that $p\divides n$. If $p=2$, then we are done. Suppose $p\neq 2$. We consider two cases.
		
		\noindent\textbf{Case 1.} $p\equiv 1 \pmod 4$. By Lemma \ref{lp}, $-1\in Q_{p_t}$ for $t\geq 1$. Therefore $0<r<s$. Assume that $p_1=p$. Note that $\chi_0(Q_{p^t})=|Q_{p^t}|$. Therefore for $a=(a_1,0,\hdots, 0)\in \Zl_n$, we have $$\ld_a= 2 |Q_{p_2^{t_2}}|\hdots |Q_{p_s^{t_s}}| \chi_{a_1}\left(Q_{p_1^{t_1}}\right).$$
		Hence, the corresponding eigenvalue of the discriminant $P$ is given by
		\begin{align*}
			\mu_a&=\frac{1}{| Q_{p_1^{t_1}}|} \chi_{a_1}\left(Q_{p_1^{t_1}}\right).
		\end{align*}
		Choose $b_1\in p_1^{t_1-1}\Zl_{p_1}\setminus \{0\}$ such that $\left( \frac{b_1}{p_1^{t_1-1}}\Big/ p_1\right)=1$, and let $b=(b_1, 0, \hdots, 0)$. Using Lemmas \ref{kkk}, \ref{kk} and \ref{quc1} in the proceeding expression of $\mu_a$, we have  $$\mu_b=\frac{1}{1+\sqrt{p_1}}.$$ Choose   $c_1\in p_1^{t_1-1}\Zl_{p_1}\setminus \{0\}$ such that $\left( \frac{c_1}{p_1^{t_1-1}}\Big/ p_1\right)=-1$, and let $c=(c_1, 0, \hdots, 0)$. Then, $$\mu_c=\frac{1}{1-\sqrt{p_1}}.$$
		If $\Gn$ is periodic, then by Lemma \ref{m1}, the discriminant eigenvalues $\mu_b$ and $\mu_c$ are real parts of some roots of unity. Consider the polynomial $f(x)=(x-\mu_b)(x-\mu_c)$. Then,
		$$f(x)=x^2-\frac{2}{1-p_1}x + \frac{1}{1-p_1} \in \Ql[x].$$ Therefore by Lemma \ref{m2}, $(2x)^2f\left(\frac{x+x^{-1}}{2}\right)$ is a product of some cyclotomic polynomials. Thus, we have $(2x)^2f\left(\frac{x+x^{-1}}{2}\right)\in \Zl[x].$ Now $$(2x)^2f\left(\frac{x+x^{-1}}{2}\right)= x^4-\frac{4}{1-p_1}x^3 + \frac{6-2p_1}{1-p_1}x^2- \frac{4}{1-p_1}x+1.$$
		Since $p_1\equiv 1$ (mod $4$), we see that $(2x)^2f\left(\frac{x+x^{-1}}{2}\right)\in \Zl[x]$ if and only if $p_1=5$. Thus, the only possible prime factor of $n$ is $5$.
		
		\noindent\textbf{Case 2.} $p\equiv 3 \pmod 4$. Assume that $p_s=p$. For $a=(0,0,\hdots, a_s)\in \Zl_n$, we have $$\ld_a= |Q_{p_1^{t_1}}|\cdots |Q_{p_{s-1}^{t_{s-1}}}| \left[ \chi_{a_s}\left(Q_{p_s^{t_s}}\right)+ \compconj{\chi_{a_s}\left(Q_{p_s^{t_s}}\right)} \right].$$
		Therefore, the corresponding eigenvalue of $P$ is given by $$\mu_a= \frac{1}{2|Q_{p_s^{t_s}}|}\left[ \chi_{a_s}\left(Q_{p_s^{t_s}}\right)+ \compconj{\chi_{a_s}\left(Q_{p_s^{t_s}}\right)} \right].$$
		Now choose $a_s\in p_s^{t_s-1}\Zl_{p_s}\setminus \{0\}$ such that $\left( \frac{a_s}{p_s^{t_s-1}}\Big/ p_s\right)=1$ and let $a=(0,\hdots, 0,a_s)$. Then, $$\mu_a=\frac{1}{1-p_s}.$$
		Consider the monic rational polynomial $f(x)=x-\mu_a$. If $\Gn$ is periodic, then by Lemma \ref{m2}, $(2x)f\left(\frac{x+x^{-1}}{2}\right)\in \Zl[x]$. Since $p_s\equiv 3$ (mod $4$), we must have $p_s=3$.
		
		Now we consider the case $r=s$. We find that the regularity of the graph $\Gn$ is $|Q_{p_1^{t_1}}|\cdots |Q_{p_s^{t_s}}|$. Further, the eigenvalues of the adjacency matrix of $\Gn$ are given by $$	\ld_a=\chi_{a_1}\left(Q_{p_1^{t_1}}\right)\hdots \chi_{a_r} \left(Q_{p_s^{t_s}}\right),$$ where $a=(a_1,\hdots,a_s)\in\Zl_n$ and $a_i\in Z_{p_i^{t_i}}$ for $1\leq i\leq s$. Let $p$ be a prime such that $p\divides n$. If $p=2$, we are done. As $r=s$, we have that $-1\in Q_{p_i^{t_i}} $ for $1\leq i\leq s$. Therefore by Lemma \ref{lp}, $p \equiv 1 \pmod 4$. Now proceeding as in Case 1, we find that $p=5$.
		
		Thus, we find that if $\Gn$ is periodic and $p$ is a prime factor of $n$, then $p\in\{2,3,5\}$.
	\end{proof}
	\begin{lema}\label{k11}
		Let $n=p^t$, where $p$ is a prime number and $t$ is a positive integer. Then, $\Gn$ is periodic if and only if $p\in\{2,3,5\}$.
	\end{lema}
	\begin{proof}
		The necessary part follows from Theorem \ref{main}. Note that the discriminant matrix $P$ of $\Gn$ is $\frac{1}{k}A$, where $k=|V_n|$ and $A$ is the adjacency matrix of $\Gn$. It is clear that  $\spec_P(\mathcal{G}_{\Zl_{2}})=\{\pm1\}$ and $\spec_P(\mathcal{G}_{\Zl_{4}})=\{\pm1,0\}$. We now use Theorems \ref{quc1} and \ref{quc2} to determine $\spec_P(\mathcal{G}_{\Zl_{2^t}})$ for $t\geq 3$. By Lemma \ref{lp}, $-1\notin Q_{2^t}$ for $t\geq 3$. Therefore by Theorem \ref{quc2}, for $a\in\Zl_{2^t}$, we have $$\ld_a=\chi_a(Q_{2^t})+\compconj{\chi_a(Q_{2^t})}=2\real(\chi_a(Q_{2^t})).$$ We obtain the value of $\chi_a(Q_{2^t})$ from Theorem \ref{quc1}. Thus, we have 
		$$\ld_a= \left\{ \begin{array}{ll}
			2^{t-2}\cos\left(\frac{a\pi}{2^{t-1}}\right) &\mbox{ if }	a\in\left\{2^{t-3}, 3\cdot 2^{t-3}, 5\cdot2^{t-3}, 7\cdot 2^{t-3} \right\} \\ 
			2^{t-2}\real\left((-\iu)^{\frac{a}{2^{t-2}}}\right) &\mbox{ if } a\in \left\{0, 2^{t-2}, 2^{t-1}, 3\cdot2^{t-2} \right\}\\ 
			0 &\textnormal{ otherwise.}
		\end{array}\right. $$
		Therefore, $\spec_A(\mathcal{G}_{\Zl_{2^t}})=\left\{\pm 2^{t-2},\pm \frac{2^{t-2}}{\sqrt{2}},0\right\}$. As $k=2|Q_{2^t}|=2^{t-2}$, we finally have $$\spec_P(\Ga)=\left\{\pm1, \pm \frac{1}{\sqrt{2}}, 0\right\}.$$
		
		Now we consider $p=3$. In this case also, $-1\notin Q_{3^t}$ for $t\geq 1$. Therefore, $\ld_a=2\real(\chi_a(Q_{3^t}))$ for $a\in \Zl_{3^t}$. Calculating the value of $\chi_a(Q_{3^t})$ from Theorem \ref{quc1} and Lemma \ref{kk}, we find that $\spec_A(\Gb)=\{2\cdot 3^{t-1},0,-3^{t-1}\}$. As $k=2\cdot3^{t-1}$, we have $$\spec_P(\Gb)=\left\{1,0,-\frac{1}{2}\right\}.$$
		
		Similarly, we have 
		$$\spec_P(\Gc)=\left\{1,~-\frac{1}{4}\pm \frac{\sqrt{5}}{4},~0\right\}.$$

		Thus by Theorem \ref{ls}, the quadratic unitary Cayley graph  $\Gp$ is periodic for $p\in\{2,3,5\}$.
	\end{proof}
	\begin{lema}\label{2p}
		Let $p$ be a prime number such that $p\in\{3,5\}$ and $t$ is a positive integer. Then, the graphs $\mathcal{G}_{\Zl_{2\cdot p^t}}$ and $\mathcal{G}_{\Zl_{4\cdot p^t}}$ are periodic.
	\end{lema}
	\begin{proof}
		Proceeding as in the preceding lemma, we find
		\begin{align*}
			\spec_P(\mathcal{G}_{\Zl_{2\cdot 3^t}})&=\left\{\pm1,\pm\frac{1}{2},0\right\}, \\ \spec_P(\mathcal{G}_{\Zl_{4\cdot 3^t}})&=\left\{\pm1,\pm\frac{\sqrt3}{2},\pm\frac{1}{2},0\right\}~\text{and} \\
			\spec_P(\mathcal{G}_{\Zl_{2\cdot 5^t}})&=\left\{\pm1,~\pm \frac{1}{4}\pm \frac{\sqrt{5}}{4},~0\right\}=\spec_P(\mathcal{G}_{\Zl_{4\cdot 5^t}}).			
		\end{align*}
		Therefore by Theorem \ref{ls}, the graph $\Gn$ is periodic for $n\in\{2\cdot3^t,4\cdot3^t,2\cdot5^t,4\cdot5^t\}.$
	\end{proof}
	\begin{lema}\label{k12}
		Let $n=2^\alpha3^\beta$ for some positive integers $\alpha$ and $\beta$. Then, $\Gn$ is periodic.
	\end{lema}
	\begin{proof}
		If $\alpha\leq 2$, then by Lemma \ref{2p}, the graph $\Gn$ is periodic. Now let $\alpha\geq 3$. The degree of the graph $\Gn$ is $2|Q_{2^\alpha}||Q_{3^\beta}|.$ By Lemma \ref{kkk}, we have $2|Q_{2^\alpha}||Q_{3^\beta}|=2^{\alpha -2}3^{\beta-1}$. We write
		$\Zl_n=\Zl_{2^\alpha}\times \Zl_{3^\beta}.$ Let $a=(a_1,a_2)\in\Zl_n$, where $a_1\in \Zl_{2^\alpha}$ and $a_2\in\Zl_{3^\beta}$. Then, adjacency eigenvalues of $\Gn$ are given by $$\ld_a=\left[\chi_{a_1}\left(Q_{2^\alpha}\right) \chi_{a_2}\left(Q_{3^\beta}\right)+\compconj{\chi_{a_1}\left(Q_{2^\alpha}\right)\chi_{a_2}\left(Q_{3^\beta}\right)}\right].$$
		Now we consider a few cases depending on the values of $a_1$ and $a_2$.
		
		\noindent\textbf{Case 1.} $a_1\in \left\{2^{\alpha-3}, 3\cdot 2^{\alpha-3}, 5\cdot2^{\alpha-3}, 7\cdot 2^{\alpha-3} \right\}$ and $a_2\in 3^{\beta -1}\Zl_3\setminus\{0\}$. If $a_2=3^{\beta -1}$, then by Theorem~\ref{quc1}, $$\ld_a=-2^{\alpha-3}3^{\beta -1}\left(\cos\frac{a_1\pi}{2^{\alpha-1}}+\sqrt{3}\sin\frac{a_1\pi}{2^{\alpha-1}}\right).$$ Then, the corresponding eigenvalue of $P$ is given by \begin{align*}
			\mu_a=\frac{\ld_a}{k}=-\frac{1}{2}\cos\frac{a_1\pi}{2^{\alpha-1}}-\frac{\sqrt{3}}{2}\sin\frac{a_1\pi}{2^{\alpha-1}}=\cos{\left(\frac{2^\alpha+3a_1}{3\cdot2^{\alpha-1}}\pi\right)}.
		\end{align*}	
		If $a_2=2\cdot3^{\beta-1}$, then $$\ld_a=2^{\alpha-3}3^{\beta -1}\left(-\cos\frac{a_1\pi}{2^{\alpha-1}}+\sqrt{3}\sin\frac{a_1\pi}{2^{\alpha-1}}\right).$$ Then, the corresponding eigenvalue of $P$ is given by \begin{align*}
			\mu_a&=\cos{\left(\frac{2^\alpha-3a_1}{3\cdot2^{\alpha-1}}\pi\right)}.
		\end{align*} 
		
		\noindent\textbf{Case 2.} $a_1\in \left\{0, 2^{\alpha-2}, 2^{\alpha-1}, 3\cdot2^{\alpha-2} \right\}$ and $a_2\in3^{\beta-1}\Zl_3$.  If $a_2=3^{\beta-1}$, then $$\ld_a=2^{\alpha -4}3^{\beta-1}\left[(-\iu)^{\frac{a_1}{2^{\alpha-2}}}(-1-\iu\sqrt{3})+(\iu)^{\frac{a_1}{2^{\alpha-2}}}(-1+\iu\sqrt{3})\right].$$ Therefore, the eigenvalues of $P$ are $\pm\frac{1}{2}, \pm\frac{\sqrt{3}}{2}$.
		
		Similarly, if $a_2=2\cdot 3^{\beta-1}$, then also the eigenvalues of $P$ are $\pm\frac{1}{2}, \pm\frac{\sqrt{3}}{2}$.
		
		\noindent\textbf{Case 3.} $a_1\in 2^{\alpha -3}\Zl_{2^\alpha}$ and $a_2=0$. In this case, we find that the eigenvalues of $P$ are $\pm1, \pm\frac{1}{\sqrt{2}}$ and $0$.
		
		\noindent\textbf{Case 4.} Either $a_1\notin 2^{\alpha -3}\Zl_{2^\alpha}$ and $a_2\in 3^{\beta -1}\Zl_3$ or $a_1\in \Zl_{2^\alpha}$ and $a_2\notin 3^{\beta -1}\Zl_3$. In this case, it is clear that $\ld_a=0$, and hence $\mu_a=0$.
		
		In all the cases, we see that $\spec_P(\Gn)\subset\Re$. Therefore by Lemma \ref{m1}, the result follows.	
	\end{proof}
	\begin{lema}
		Let $n=2^\alpha5^\gamma$, where $\alpha$ and $\gamma$ be two positive integers. Then $\Gn$ is periodic if and only if $\alpha\in \{1,2\}$.
	\end{lema}
	\begin{proof}
		Let $\alpha\geq 3$. In this case, the regularity $k$ of the graph is given by $k=2|Q_{2^\alpha}||Q_{5^\gamma}|=2^{\alpha-1}5^{\gamma-1}$.
		
		Then by Theorem \ref{quc2}, the eigenvalue $\mu_a$ of the discriminant $P$ is given by 
		\begin{equation}\label{lm47}
			\mu_a=\frac{1}{k}\chi_{a_1}(Q_{5^\gamma})\left[\chi_{a_2}(Q_{2^\alpha}) +\compconj{\chi_{a_2}(Q_{2^\alpha}) }\right],
		\end{equation}
		where $a=(a_1,a_2)\in\Zl_{5^\gamma}\times \Zl_{2^\alpha}$.
		
		Now consider  $b_1=(5^{\gamma-1},2^{\alpha-3})$, $b_2=(2\cdot5^{\gamma-1},2^{\alpha-3})$, $b_3=(5^{\gamma-1},3\cdot2^{\alpha-3})$ and $b_4=(2\cdot5^{\gamma-1},3\cdot2^{\alpha-3})$ in $\Zl_{5^\gamma}\times \Zl_{2^\alpha}$.		
		From \eqref{lm47}, we have 
		$$\mu_{b_1}=\frac{1}{\sqrt{2}(1+\sqrt{5})},~
		\mu_{b_2}=\frac{1}{\sqrt{2}(1-\sqrt{5})},~
		\mu_{b_3}=-\frac{1}{\sqrt{2}(1+\sqrt{5})}~\text{and}~
		\mu_{b_4}=-\frac{1}{\sqrt{2}(1-\sqrt{5})}.$$
		If $\Gn$ is periodic, then $\mu_{b_1}$, $\mu_{b_2}$, $\mu_{b_3}$ and $\mu_{b_4}$ must be the real parts of some roots of unity. Now, consider the polynomial $f(x)=(x-\mu_{b_1})(x-\mu_{b_2})(x-\mu_{b_3})(x-\mu_{b_4})$. Observe that $f(x)$ is a monic polynomial in $\Ql[x]$. Then by Lemma \ref{m2}, $(2x)^4f\left(\frac{x+x^{-1}}{2}\right)$ must be in $\Zl[x]$. However, $$(2x)^4f\left(\frac{x+x^{-1}}{2}\right)=(x^2+1)^4-\frac{3}{2}x^2(x^2+1)^2+\frac{1}{4}x^4\notin\Zl[x],$$ a contradiction. Therefore for $\alpha\geq 3$, the graph $\Gn$ is not periodic. Hence by Lemma \ref{2p}, the result follows.	
	\end{proof}
	\begin{lema}\label{lastp}
		Let $n=2^\alpha3^\beta 5^\gamma$, where $\alpha$, $\beta$ and $\gamma$ are nonnegative integers such that $\beta\geq1$ and $\gamma\geq 1$. Then, $\Gn$ is not periodic. 
	\end{lema}
	\begin{proof}
		First, assume that $\alpha\geq1$. The eigenvalue $\mu_a$ of the discriminant $P$ of $\Gn$ is given by 
		\begin{equation}\label{lm48}
			\mu_a=\frac{1}{k}\chi_{a_1}(Q_{5^\gamma})\left[\chi_{a_2}(Q_{2^\alpha})\chi_{a_3}(Q_{3^\gamma}) +\compconj{\chi_{a_2}(Q_{2^\alpha})\chi_{a_3}(Q_{3^\gamma}) }\right],
		\end{equation}
		where $k=2|Q_{2^\alpha}||Q_{3^\beta}||Q_{5^\gamma}|$ and $a=(a_1,a_2,a_3)\in\Zl_{5^\gamma}\times \Zl_{2^\alpha}\times\Zl_{3^\beta}.$
		Thus for $a=(5^{\gamma-1}, 0,3^{\beta-1})$, we have $\mu_{a}=\frac{1}{8}+\frac{\sqrt{5}}{8}$.  Similarly, for $\alpha=0$ and $a=(5^{\gamma-1},3^{\beta-1})\in\Zl_{5^\gamma}\times\Zl_{3^\beta}$, we have $\mu_{a}=\frac{1}{8}+\frac{\sqrt{5}}{8}$. Therefore by Theorem \ref{ls}, the graph $\Gn$ is not periodic.
	\end{proof}
	The following theorem is obtained by combining the preceding lemmas.
	\begin{theorem}\label{main1}
		The quadratic unitary Cayley graph $\Gn$ is periodic if and only if  $n$ is either $2^\alpha3^\beta$ or  $2^\delta5^\gamma$ for some nonnegative integers $\alpha$, $\beta$, $\gamma$ and $\delta$ such that $\alpha+\beta\neq 0$,  $\gamma\geq1$ and $0\leq\delta\leq2$. 
	\end{theorem}
	Recall that  $\spec_A(\mathcal{G}_{\Zl_{3^\beta}})=\{2\cdot 3^{\beta-1},0,-3^{\beta-1}\}$ and $\spec_A(\mathcal{G}_{\Zl_{2^\alpha}})=\left\{\pm 2^{\alpha-2},\pm \frac{2^{\alpha-2}}{\sqrt{2}},0\right\}$ for $\alpha \geq 3$. Thus, $\mathcal{G}_{\Zl_{3^\beta}}$ is integral while $\mathcal{G}_{\Zl_{2^\alpha}}$, for $\alpha\geq 3$, is nonintegral. However, both $\mathcal{G}_{\Zl_{3^\beta}}$ and $\mathcal{G}_{\Zl_{2^\alpha}}$, for $\alpha \geq 3$, are periodic. Thus, we find infinitely many integral as well as nonintegral graphs that are periodic.
	
	In Lemmas \ref{k11}, \ref{2p} and \ref{k12}, we explicitly determined the eigenvalues of $P$ for the periodic graph $\Gn$. By Lemma \ref{evu} and Corollary \ref{periodic}, one can determine the period of $\Gn$. For example, we show the calculation details to obtain the period of $\Gn$ for $n=4\cdot 5^t~(t\geq1)$. By Lemma \ref{2p}, we have
	$$\spec_P(\Gn)=\left\{\pm1,~\pm \frac{1}{4}\pm \frac{\sqrt{5}}{4},~0\right\}.$$
	Consider the polynomial $f(x)=x^2-\frac{x}{2}-\frac{1}{4}$ having roots $ \frac{1}{4}+\frac{\sqrt{5}}{4}$ and $ \frac{1}{4}-\frac{\sqrt{5}}{4}$.  Let $w=e^{\iu\frac{\pi}{5}}$ and $z_1=\frac{w+w^{-1}}{2}$. Observe that $z_1>0$ and $f(z_1)=0$. Therefore, $z_1= \frac{1}{4}+\frac{\sqrt{5}}{4}$, that is, $\cos \frac{\pi}{5}=\frac{1}{4}+\frac{\sqrt{5}}{4}$.  Now, considering $w=e^{\iu\frac{3\pi}{5}}$, we find  $\cos \frac{3\pi}{5}=\frac{1}{4}-\frac{\sqrt{5}}{4}$.
	
	In a similar manner, considering the polynomial  $f(x)=x^2+\frac{x}{2}-\frac{1}{4}$, we find that $\cos \frac{2\pi}{5}=-\frac{1}{4}+\frac{\sqrt{5}}{4}$ and $\cos \frac{4\pi}{5}=-\frac{1}{4}-\frac{\sqrt{5}}{4}$. Thus, we have
	\begin{equation*}\label{spp1}
		\spec_P(\Gn)=\left\{\cos 0,~ \cos \pi, ~\cos\frac{\pi}{5},~ \cos\frac{3\pi}{5},~ \cos\frac{2\pi}{5},~ \cos\frac{4\pi}{5}, ~\cos\frac{\pi}{2} \right\}.\end{equation*}
	Therefore by Lemma \ref{evu}, 
	$$\spec_U(\Gn)=\left\{\pm 1,~e^{\pm\iu\frac{\pi}{5}},~e^{\pm\iu\frac{3\pi}{5}},~e^{\pm\iu\frac{2\pi}{5}},~e^{\pm\iu\frac{4\pi}{5}},~e^{\pm\iu\frac{\pi}{2}}\right\}.$$
	Hence by Corollary \ref{periodic}, the period of $\Gn$ is $20$ for $n=4\cdot 5^t$. The calculations for the other cases are omitted. The next theorem provides the period of $\Gn$. For nonperiodic graphs, we denote the period as $\infty$. 
	\begin{theorem}\label{gnp}
		The period $\tau$ of the quadratic unitary Cayley graph $\Gn$ is given by 	$$\tau= \left\{ \begin{array}{cl}
			2 &\mbox{ if }	n=2 \\ 
			4 &\mbox{ if }	n=4 \\ 
			8 &\mbox{ if }	n=2^\alpha~(\alpha\geq3) \\ 
			12 &\mbox{ if }	n\in\{3^\beta,2\cdot3^\beta,4\cdot3^\beta\} \\ 
			20 &\mbox{ if }	n\in\{5^\gamma,2\cdot5^\gamma,4\cdot5^\gamma\} \\ 
			24 &\mbox{ if }	n=2^\alpha3^\beta~ (\alpha\geq 3) \\ 
			\infty &\textnormal{ otherwise,}
		\end{array}\right.$$ 
		where $\alpha,~\beta$ and $\gamma$ are positive integers.
	\end{theorem}
	From Theorems \ref{main1} and \ref{gnp}, we find that the graphs in Figures \ref{fig:sub1} and \ref{fig:sub2} are periodic, with periods $12$ and  $20$, respectively.
	\section{Perfect state transfer}
	Perfect state transfer is not just a theoretical curiosity but also a practical necessity for advancing quantum technologies, including quantum algorithm design, quantum meteorology, communication efficiency and numerous other applications. Therefore, discovering new graphs that exhibit perfect state transfer is important. In this section, we characterize perfect state transfer on quadratic unitary Cayley graphs. In \cite{bhakta1}, the authors proved that periodicity is a necessary condition for the occurrence of perfect state transfer on a vertex-transitive graph. Note that a Cayley graph is vertex-transitive.
	\begin{theorem}\emph{\cite{bhakta1}}\label{thm2}
		Let $G$ be a vertex-transitive graph. If $G$ exhibits perfect state transfer, then it is periodic.
	\end{theorem}
	In Grover walks, the occurrence of perfect state transfer between vertex-type states has a strong connection with Chebyshev polynomials. The \emph{Chebyshev polynomial of the first kind}, denoted $T_n(x)$, is defined recursively as follows: $T_0(x)=1$, $T_1(x)=x$ and $$T_n(x)=2xT_{n-1}(x)-T_{n-2}(x) ~\text{for}~ n \geq 2.$$ It is well known that 
	\begin{equation}\label{chb}
		T_n(\cos\theta)=\cos(n\theta).
	\end{equation}
	The authors of \cite{bhakta1} provided a necessary and sufficient condition in terms of the Chebyshev polynomial of the first kind for the occurrence of perfect state transfer on circulant graphs.
	\begin{theorem}\emph{\cite{bhakta1}}\label{thl}
		Let $\mu_0,\mu_1,\hdots,\mu_{n-1}$ be the eigenvalues of the discriminant  of a circulant graph  $\cay(\Zl_n, S)$, where $\mu_j=\frac{1}{|S|}\sum_{s\in S} e^{-\frac{2\pi j s \iu }{n}}$ for $j\in\{0,\hdots, n-1\}$.   Then, perfect state transfer occurs in $\cay(\Zl_n, S)$ from a vertex $u$ to another vertex $v$ at time $\tau$ if and only if all of the following conditions hold.
		\begin{enumerate}[label=(\roman*)]
			\item $n$ is even and $u-v=\frac{n}{2}$.
			\item $T_\tau(\mu_j)=\pm1$ for $j\in\{0,\hdots,n-1\}$, where $T_n(x)$ is the Chebyshev polynomial of the first kind.
			\item $T_\tau(\mu_j)\neq T_\tau(\mu_{j+1})$ for $j\in\{0,\hdots,n-2\}.$
		\end{enumerate}
	\end{theorem}	
	The next two lemmas are easy consequences of the preceding theorem.
	\begin{lema}\emph{\cite{bhakta1}}\label{cycle}
		A cycle on $n$ vertices exhibits perfect state transfer if and only if $n$ is even.
	\end{lema}
	\begin{lema}\emph{\cite{bhakta1}}\label{complete}
		The only complete graph $K_n$ exhibiting perfect state transfer is $K_2$.
	\end{lema}
	The following theorem determines the values of $n$ for which the quadratic unitary Cayley graph $\Gn$ exhibits perfect state transfer.
	
	\begin{theorem}\label{main2}
		The quadratic unitary Cayley graph $\Gn$ exhibits perfect state transfer if and only if $n\in\{ 2^\alpha3^\beta,~10,~20\},$ where $\alpha$ and $\beta$ are nonnegative integers such that $1\leq \alpha\leq3$ and $0\leq\beta\leq1$.
	\end{theorem}
	\begin{proof}
		If the quadratic unitary Cayley graph $\Gn$ exhibits perfect state transfer, then by Theorems \ref{thm2} and \ref{main1}, we find that $n\in\{ 2^\alpha3^\beta,~5^\gamma,~2\cdot5^\gamma,~4\cdot5^\gamma\}$, where $\alpha$, $\beta$ and $\gamma$ are nonnegative integers such that $\alpha+\beta\neq0$ and $\gamma\geq1$. By Theorem \ref{thl}, if $\Gn$ exhibits perfect state transfer, then $n$ must be even. Therefore, $n\in\{2^\alpha3^\beta,~2\cdot5^\gamma,~4\cdot5^\gamma\}$, where $\alpha\geq1$, $\beta\geq0$ and $\gamma\geq 1$.
		
		For $n=2^\alpha3^\beta$ ($1\leq \alpha\leq3$ and $0\leq\beta\leq1$) and $n=10$, the unitary Cayley graph $\Gn$ is $K_2$ and the cycle $C_n$ for $n>3$. Therefore by Lemmas \ref{cycle} and \ref{complete}, the graph $\Gn$ exhibits perfect state transfer for $n=2^\alpha3^\beta$ ($1\leq \alpha\leq3$ and $0\leq\beta\leq1$) and $n=10$

		For $n=20$, by Theorem \ref{quc2}, the period of the graph is $20$. Therefore, if perfect state transfer occurs in $\Gn$ at the minimum time $\tau$, then $\tau$ must be in $\{1,\hdots,19\}$.  Now $0,\frac{1}{4}+\frac{\sqrt{5}}{4}\in \spec_P(\mathcal{G}_{\Zl_{20}})$, and therefore by Theorem \ref{thl}, $T_\tau(0)=\pm 1$ and $T_\tau\left(\frac{1}{4}+\frac{\sqrt{5}}{4}\right)=\pm 1 $, that is, $T_\tau(\cos\frac{\pi}{2})=\pm1$ and $T_\tau(\cos \frac{\pi}{5})=\pm 1$. Hence by \eqref{chb}, $\cos(\frac{\tau \pi}{2})=\pm 1$ and $\cos(\frac{\tau \pi}{5})=\pm1$. Thus, $\tau\in 2\Zl$ and $\tau\in5\Zl$, and so $\tau\in10\Zl$. Therefore, $\tau=10$.

		Using Theorems \ref{quc1} and \ref{quc2}, the discriminant eigenvalues of $\mathcal{G}_{\Zl_{20}}$ are given by
		$$\mu_j= \left\{ \begin{array}{ll}
			\cos \frac{\pi}{2} &\mbox{ if }	\text{$j$ is odd} \\ 
			\cos \frac{j\pi}{10} &\mbox{ if }	j\in\{0,2,4,6,8,10\} \\ 
			\cos \frac{(20-j)\pi}{10} &\mbox{ if }	j\in\{12,14,16,18\}.
		\end{array}\right.$$ 
		Now noting that $T_n(\cos \theta)=\cos n\theta$, we have $T_{10}(\mu_j)=\pm1$ for $j\in\{0,\hdots,19\}$ and $T_{10}(\mu_j)\neq T_{10}(\mu_{j+1})$ for $j\in\{0,\hdots,18\}$. Therefore by Theorem \ref{thl}, the graph $\mathcal{G}_{\Zl_{20}}$ exhibits perfect state transfer.

		For the remaining values of $n$, we consider the following cases. Let $\mu_j$ be the discriminant eigenvalues defined as in Theorem \ref{thl} for $j\in\Zl_n$.
		
		\noindent\textbf{Case 1.} $n=2^\alpha3^\beta$  ($\alpha\geq 4$ and $\beta=0$).  By Theorems \ref{quc1} and \ref{quc2}, we have $\mu_4=0=\mu_{5}$ for $\alpha=4$. For $\alpha\geq5$, we have $\mu_1=0=\mu_2$.

		\noindent\textbf{Case 2.} $n=2^\alpha3^\beta$ ($\alpha\geq 4$ and $\beta=1$). In this case, we find that	$\mu_a=0=\mu_{a+1}$ for $a=(4,0)\in \Zl_{2^4}\times\Zl_{3}$. If $\alpha\geq5$, we have $\mu_a=0=\mu_{a+1}$ for $a=(2,0)\in \Zl_{2^\alpha}\times\Zl_{3}$ .

		\noindent\textbf{Case 3.} $n=2^\alpha3^\beta$ ($\alpha\geq1$ and $\beta\geq2$).  In this case, $\mu_a=0=\mu_{a+1}$ for $a=(0,1)\in \Zl_{2^\alpha}\times\Zl_{3^\beta}$. 
		
		\noindent\textbf{Case 4.} $n=2\cdot5^\gamma$ $(\gamma\geq2)$ or $n=4\cdot5^\gamma$ $(\gamma\geq2)$. In this case also, $\mu_a=0=\mu_{a+1}$ for $a=(1,0)\in \Zl_{5^\gamma}\times\Zl_{2^\alpha}$, where $1\leq\alpha\leq2$.

		For each of the preceding four cases, we find some $a\in \Zl_n$ such that $T_\tau(\mu_a)=T_\tau(\mu_{a+1})$ for all  $\tau\in \Nl$. Therefore by Theorem \ref{thl}, perfect state transfer does not occur in these cases. Thus, the result follows.
	\end{proof}
	Theorem \ref{main2} shows that the graph in Figure~\ref{fig:sub1} does not exhibit perfect state transfer. However, the graph in Figure~\ref{fig:sub2} exhibits perfect state transfer from vertex $a$ to vertex $b$ at time $10$, where $|a-b|=10$.
	\subsection*{Acknowledgements}
	The first author acknowledges the support provided by the Prime Minister's Research Fellowship (PMRF), Government of India (PMRF-ID: 1903298). The second author thanks the Science and Engineering Research Board (SERB), Government of India, for supporting this work under the MATRICS Project [File No. MTR/2021/000075].


\begin{thebibliography}{10}
		
		
		
		\bibitem{quantum}
		D. Aharonov, A. Ambainis, J. Kempe, and U. Vazirani.
		\newblock Quantum walks on graphs.
		\newblock {\em Proceedings of the 33rd Annual ACM Symposium on Theory of Computing}. 50--59, 2001.
		
		\bibitem{random}
		Y. Aharonov, L. Davidovich, and N. Zagury.
		\newblock Quantum random walks.
		\newblock {\em Physical Review A}. 48(2):1687--1690, 1993.
		
		
		
		\bibitem{banerjee}
		A. Banerjee.
		\newblock Discrete quantum walks on the symmetric group.
		\newblock {\em Quantum Studies: Mathematics and Foundations}. 11:477--490, 2024.
		
		
		\bibitem{barr}
		K. Barr, T. Proctor, D. Allen, and V. Kendon.
		\newblock Periodicity and perfect state transfer in quantum walks on variants of cycles.
		\newblock {\em Quantum Information and Computation}. 14(5\&6):417--438, 2014.
		
		\bibitem{qucfirst}
		N. de Beaudrap.
		\newblock On restricted unitary Cayley graphs and symplectic transformations modulo $n$.
		\newblock {\em The Electronic Journal of Combinatorics}. 17:\#R69, 2010.
		
		\bibitem{bermond}
		J.C. Bermond, F. Comellas, and D.F. Hsu.
		\newblock  {Distributed loop computer networks: a survey}.
		\newblock {\em Journal of Parallel and Distributed Computing}. 24:2--10, 1995.
		
		\bibitem{bhakta1}
		K. Bhakta and B. Bhattacharjya.
		\newblock Grover walks on unitary Cayley graphs and integral regular graphs.
		\newblock {\em arXiv:2405.01020}. 2024.
		
		
		\bibitem{polytime}
		G. Coutinho and C. Godsil.
		\newblock Perfect state transfer is poly-time.
		\newblock {\em Quantum Information and Computation}. 17:495--502, 2017.
		
		
		\bibitem{periodic}
		C. Godsil.
		\newblock Periodic graphs.
		\newblock {\em The Electronic Journal of Combinatorics}. 18(1):\#P23, 2011.
		
		
		\bibitem{godsildct}
		C. Godsil and H. Zhan.
		\newblock Discrete-time quantum walks and graph structures.
		\newblock {\em Journal of Combinatorial Theory, Series A}. 167:181--212, 2019.
		
		\bibitem{higu}
		Y.~Higuchi, N.~Konno, I.~Sato, and E.~Segawa.
		\newblock A note on the discrete-time evolutions of quantum walk on a graph.
		\newblock {\em Journal of Math-for-Industry}. 5:103--109, 2013.
		
		\bibitem{higuchi1}
		Y.~Higuchi, N.~Konno, I.~Sato, and E.~Segawa.
		\newblock Periodicity of the discrete-time quantum walk on a finite graph.
		\newblock {\em Interdisciplinary Information Sciences}. 23(1):75--86, 2017.
		
		\bibitem{spec}
		Y.~Higuchi, N.~Konno, I.~Sato, and E.~Segawa.
		\newblock Spectral and asymptotic properties of Grover walks on crystal lattices.
		\newblock {\em Journal of Functional Analysis}. 267(11):4197–4235, 2014.
		
		\bibitem{quc}
		J. Huang.
		\newblock On the quadratic unitary Cayley graphs.
		\newblock {\em Linear Algebra and its Applications}. 644:219--233, 2022.
		
		
		\bibitem{ito}
		N.~Ito, T.~Matsuyama, and T.~Tsurii.
		\newblock Periodicity of Grover walks on complete graphs with self-loops.
		\newblock {\em Linear Algebra and its Applications}. 599:121--132, 2020.
		
		\bibitem{complex}
		A. Jeffrey.
		\newblock Complex Analysis and Applications.
		\newblock {\em Chapman \& Hall/CRC, Boca Raton, FL,
			Second Edition}. 2006.
		
		
		
		\bibitem{ucg}
		W.~Klotz and T.~Sander.
		\newblock Some properties of unitary Cayley graphs.
		\newblock {\em The Electronic Journal of Combinatorics}. 14:\#R45, 2007.
		
		
		\bibitem{kubota}
		S. Kubota.
		\newblock Combinatorial necessary conditions for regular graphs to induce periodic quantum walks.
		\newblock {\em Linear Algebra and its Applications}. 673:259--279, 2023.
		
		\bibitem{bipartite}
		S.~Kubota.
		\newblock Periodicity of Grover walks on bipartite regular graphs with at most five distinct eigenvalues.
		\newblock {\em Linear Algebra and its Applications}. 654:125--142, 2022.
		
		\bibitem{vertextype}
		S.~Kubota and E.~Segawa.
		\newblock Perefct state transfer in Grover walks between states associated to vertices of a graph.
		\newblock {\em Linear Algebra and its Applications}. 646:238--251, 2022.
		
		\bibitem{qq}
		S.~Kubota, E.~Segawa, and T.~Taniguchi.
		\newblock Quantum walks defined by digraphs and generalized Hermitian adjacency matrices.
		\newblock {\em Quantum Information Processing}. 20:95, 2021.
		
		\bibitem{bethetrees}
		S.~Kubota, E.~Segawa, T.~Taniguchi, and Y.~Yoshie.
		\newblock Periodicity of Grover walks on generalized Bethe trees.
		\newblock {\em Linear Algebra and its Applications}. 554:371--391, 2018.
		
		\bibitem{mixedpaths}
		S. Kubota, H. Seikdo, and H. Yata.
		\newblock Periodicity of quantum walks defined by mixed paths and mixed cycles.
		\newblock {\em Linear Algebra and its Applications}. 630:15--38, 2021.		
		
		
		\bibitem{liu}
		X.G. Liu and S.M. Zhou.
		\newblock Quadratic unitary Cayley graphs of finite commutative rings.
		\newblock {\em Linear Algebra and its Applications}. 479:73--90, 2015.
		
		
		\bibitem{gauss}
		B. Luong.
		\newblock Fourier Analysis on Finite Abelian Groups.
		\newblock {\em Birkh\"{a}user, Boston}. 2009.	
		
		
		\bibitem{galois}
		P. Morandi.
		\newblock Field and Galois Theory.
		\newblock {\em Springer, New York}. 1996.
		
		\bibitem{zucker}
		I. Niven, H.S. Zuckerman, and H.L. Montgomery.
		\newblock An Introduction to the Theory of Numbers.
		\newblock {\em John Wiley \& Sons}. 1991.
		
		
		\bibitem{pal}
		H. Pal and B. Bhattacharjya.
		\newblock  {Perfect state transfer on gcd-graphs}.
		\newblock {\em Linear Multilinear Algebra}. 65(11):2245--2256, 2017.
		
		
		\bibitem{sarkar2}
		R.S. Sarkar and B. Adhikari.
		\newblock  Discrete-time quantum walks on Cayley graphs of Dihedral groups using generalized Grover coins.
		\newblock {\em Quantum Information Processing}. 23(5):172, 2024.
		
		\bibitem{sarkar1}
		R.S. Sarkar, A. Mandal, and  B. Adhikari.
		\newblock  Periodicity of lively quantum walks on cycles with generalized Grover coin.
		\newblock {\em Linear Algebra and its Applications}. 604:399--424, 2020.
		
		\bibitem{singh}
		S. Singh, B. Adhikari, S. Dutta, and D. Zueco.
		\newblock  Perfect state transfer on hypercubes and its implementation using superconducting qubits.
		\newblock {\em Physical Review A}. 102(6):062609, 2020.
		
		\bibitem{soffia}
		A. Soff\'ia and C. Godsil.
		\newblock  {On state transfer in {C}ayley graphs for abelian groups}.
		\newblock {\em Quantum Information Processing}. 22(1):24, 2023.
		
		\bibitem{pst_symm}
		M. \v{S}tefa\v{n}\'ak and S. Skoup\'y.
		\newblock Perfect state transfer by means of discrete-time quantum walk search algorithms on highly symmetric graphs.
		\newblock {\em Physical Review A}. 94(2):022301, 2016.
		
		
		\bibitem{dqw3}
		M. \v{S}tefa\v{n}\'ak and S. Skoup\'y.
		\newblock Perfect state transfer by means of discrete-time quantum walk on complete bipartite graphs.
		\newblock {\em Quantum Information Processing}. 16(3):72, 2017.
		
		
		

		
		
		\bibitem{rep}
		B. Steinberg.
		\newblock Representation Theory of Finite Groups: An Introductory Approach.
		\newblock {\em Springer, New York}. 2009.
		
		
		
		\bibitem{paley_number}
		C.H.~Yip.
		\newblock On the clique number of {P}aley graphs of prime power order.
		\newblock {\em Finite Fields and their Applications}. 77:101930, 2022.
		
		
		\bibitem{yoshie1}
		Y.~Yoshie.
		\newblock Characterizations of graphs to induce periodic Grover walk.
		\newblock {\em Yokohama Mathematical Journal}. 63:9--23, 2017.
		
		\bibitem{yoshie2}
		Y.~Yoshie.
		\newblock Odd-periodic Grover Walks.
		\newblock {\em Quantum Information Processing}. 22:316, 2023.
		
		\bibitem{yoshie3}
		Y.~Yoshie.
		\newblock Periodicity of Grover walks on distance-regular graphs.
		\newblock {\em Graphs and Combinatorics}. 35:1305--1321, 2019.	
		
		\bibitem{zhan}
		H. Zhan.
		\newblock An infinite family of circulant graphs with perfect state transfer in discrete quantum walks.
		\newblock {\em Quantum Information Processing}. 18(12):1--26, 2019.
		
		
		
		
		
	\end{thebibliography}
\end{document}